\documentclass[reqno,a4paper]{amsproc} 
\usepackage{verbatim,graphicx,paralist}
\usepackage{hyperref}

\newcommand{\mnote}[1]{$\ll$\textsf{#1 --Manor}$\gg$\marginpar{\tiny\bf MM}}
\DeclareMathOperator{\dist}{dist}

\DeclareMathOperator*{\Avg}{\mathbb E}
\DeclareMathOperator{\lca}{lca}
\newcommand{\Lip}{\mathrm{Lip}}
\newcommand{\e}{\varepsilon}
\newcommand{\Matousek}{Matou\v{s}ek}

\newcommand{\X}{\mathcal{X}}
\newcommand{\Y}{\mathcal{Y}}
\newcommand{\HH}{\mathcal{H}}
\newcommand{\MET}{\mathcal{MET}}

\newcommand{\tb}{{|\hspace{-0.9pt}|\hspace{-0.9pt}|}}
\newtheorem{theorem}{Theorem}
\newtheorem{lemma}{Lemma}

\newtheorem{proposition}[lemma]{Proposition}

\theoremstyle{remark}
\newtheorem{definition}{Definition}
\newtheorem{remark}{Remark}
\newtheorem{question}{Question}

\begin{document}
\title{Metric Dichotomies}
\author{Manor Mendel}
\address{The Open University of Israel}
\email{mendelma@gmail.com}

\begin{abstract}
These are notes from talks given at  ICMS, Edinburgh, 4/2007 (``Geometry and Algorithms  workshop") and
at Bernoulli Center, Lausanne 5/2007 (``Limits of graphs in group theory and computer science").
We survey the following type of dichotomies exhibited by certain classes $\X$ of finite metric spaces: For every host space $H$, either all metrics in $\X$ embed almost isometrically in $H$, or the distortion of embedding some metrics of $\X$ in $H$ is unbounded.
\end{abstract}

\maketitle

\section{Problem statement and motivation}
\label{sec:intro}

In these notes we examine dichotomy phenomena exhibited by certain classes  $\X$ of finite metric spaces. When attempting to embed the metrics in $\X$ in any given host spaces $H$, either all of them embed almost isometrically, or there are some metrics in $\X$ which are very poorly embedded in $H$. To make this statement precise we define the distortion of metric embeddings.

Given a mapping between metric spaces $f:X\to H$, define the Lipschitz norm of $f$ to be $\|f\|_{\Lip}=\sup_{x\ne y} d_H(f(x),f(y))/d_X(x,y)$. The distortion of injective mapping $f$ is defined as 
$\dist(f)=\|f\|_{\Lip} \cdot \|f^{-1}\|_{\Lip}$, where 
$f^{-1}$ is defined on $f(X)$.
The ``least distortion" in which $X$ can be embedded in $H$ is defined as
 $c_H(X)=\inf\{\dist(f)|\; f:X\to H\}$. This is a measure of the faithfulness
 possible when representing $X$ using a subset of $H$. 

We formalize the discussion above as follows:
\begin{definition}[\textbf{Qualitative Dichotomy}]
 \label{def:qual-dichotomy}
A class of \emph{finite} metric spaces $\mathcal{X}$ has the qualitative dichotomy property if
for any host space $H$, either 
\begin{compactitem}
\item $\sup_{X\in\X} c_{H}(X)=1$; or
\item $\sup_{X\in\X}c_{H}(X)=\infty$.
\end{compactitem}
\end{definition}

\begin{remark}
As defined in Def.~\ref{def:qual-dichotomy}, the dichotomy is with respect to all metric spaces as hosts.
It is possible to extend the definition 
to be with respect to all \emph{sets of metric spaces} as hosts.  That is, for a set of metric spaces $\HH$, define
$c_{\HH}(X)=\inf _{H\in \HH} c_H(X)$, and replace the use of ``$c_H(X)$" in Def.~\ref{def:qual-dichotomy} with
``$c_{\HH}(X)$".
This extension, however, is inconsequential and the two definitions are equivalent. This follows from \emph{the proof} of Theorem~1.6 in~\cite{MN-cotype-full}, which implies that for any set of metric spaces $\HH$, there exists a metric $\hat H$, such that for any finite metric space $X$, $c_{\hat H}(X)=c_{\HH}(X)$.
\end{remark}

A dichotomy theorem for $\mathcal X$ can be interpreted as a 
form of rigidity of $\mathcal X$: Small deformations of all the spaces in $\mathcal X$ is impossible.

We will also be interested in stronger dichotomies --- of a quantitative nature --- in which the
unboundedness condition of the  
distortion is replaced with quantitative estimates on the rate in which it tends to infinity as a function of the size of the metric space.
I.e., by asymptotic lower bounds on the sequence 
\[ D_N(H, \X)=\sup\{c_{H}(X):\ X\in \X, \ |X|\le N\}. \]

The question of identifying such dichotomies was first explicitly raised  by Arora \emph{et.~al.}~\cite{ALNRRV}. 
They were motivated by a question from the theory of combinatorial approximation algorithms, where
bounded distortion embeddings have become a basic tool.
When dealing with algorithmically hard problem on a metric data
$X\in \X$, some algorithms first embed  $X$ into
a better understood metric space $H$, $e:X\to H$, and then solve the algorithmic problem on $e(X)$. 
This approach is used, for example,
in~\cite{LLR,Bartal-embed,GKR,Feige,FRT}.
For this approach to work:
\begin{compactenum}
\item  $H$ should be simple enough to make the algorithmic problem tractable.
\item $e(X)$ should be close to $X$.  
\end{compactenum}
Metric dichotomies draw limits on this approach when ``closeness" is measured in terms of the distortion. Dichotomy means that either $H$ already (essentially) contains $\X$, and therefore cannot be understood better than $\X$, or $H$ does not approximate some metrics in $\X$ very well. The algorithmic point of view also motivates the interest in 
quantitative dichotomies: When dealing with finite objects, slowly growing approximation ratios
are also useful, and can be ruled out by quantitative dichotomies.

{\Matousek}~\cite{Mat-BD} studied a closely related notion, which he called 
\emph{bounded distortion (bd-) Ramsey}. 
Simplifying his definitions a bit, a class of finite metric space $\X$ is called bd-Ramsey,
If for every $K>1$, $\e>0$, and $X\in\X$, there exists $Y\in\X$ such that
for any host space $H$, and any embedding $f:Y\to H$, if $\dist(f)\le K$,
then there exists $g:X\to Y$ such that 
$\dist(g)\le 1+\e$, and $\dist(f|_{g(X)})\le 1+\e$.

As observed in~\cite{ALNRRV}, the bd-Ramsey property implies 
qualitative dichotomy. 
\begin{proposition}
If a class of finite metric spaces $\X$ is bd-Ramsey
then it has the qualitative dichotomy property.
\end{proposition}
\begin{proof}
Fix a host space $H$, and suppose that 
\begin{equation} \label{eq:bounded}
\sup_{Y\in\X}c_H(Y)<\infty.
\end{equation}
Fix $X\in \X$, and $\e\in (0,1/2)$, and let $K=1+ \sup_{Y\in\X}c_H(Y)$. Pick
$Y\in\X$ that satisfies the bd-Ramsey condition. 
By~\eqref{eq:bounded}, there exists $f:Y\to H$ such that $\dist(f)\le K$.
By the bd-Ramsey property, there exists $g:X\to Y$ such that
$\dist(g)\le 1+\e$, and $\dist(f|_{g(X)})\le 1+\e$, and so $c_H(X)\le
\dist(g)\cdot \dist(f|_{g(X)}) \le 1+3\e$. Since this is true for any $\e\in(0,1/2)$, we conclude that $c_H(X)=1$. As this is true for any $X\in\X$, we conclude that $\sup_{X\in\X}c_H(X)=1$.
\end{proof}

\begin{remark}
All the dichotomies results in this note 
are actually bd-Ramsey results. 
\end{remark}

{\Matousek}'s study of bd-Ramsey phenomena~\cite{Mat-BD} is partially motivated by a general theme in the geometric theory of Banach spaces to translate notions and results from the linear theory of finite dimensional Banach spaces to finite metric spaces. 
One such example is a theorem of Maurey, Pisier, and Krivine~\cite{MP-type-cotype,Krivine} 
(see also~\cite{Maurey-survey} and~\cite[Ch.~12]{BL})  which implies that if a normed space $H$ contains linear images of $\ell_p^n$ for any $n$ with uniformly bounded distortion, then $H$ contains linear images of $\ell_p^n$ for any $n$ almost isometrically. More precisely, For every $t\in \mathbb N$, $\e>0$, $K\ge 1$, and $p\in [1,\infty]$, there exists $n=n(t,\e\,K,p)$ such that if there exists a linear mapping $f:\ell_p^n \to H$, with $\dist(f)\le K$, then there exists a linear mapping $g:\ell_p^t \to \ell_p^n$ such that both $\dist(g)\le 1+\e$, and 
$\dist(f|_{g(\ell_p^t)})\le 1+\e$. The bd-Ramsey property is a similar property, without the linear structure.

\medskip
When studying metric dichotomy for a given class $\X$ of metric spaces, it is beneficial to work with a structured dense subclass $\Y\subset \X$.

\begin{proposition} \label{prop:dense}
Suppose that $\Y \subset \X$ and $\Y$ is dense in $\X$, i.e., for every
$X\in\X$, $c_\Y(X)=1$. Then if $\Y$ has a metric dichotomy (either qualitative or quantitative) then $\X$ has the same dichotomy.
\end{proposition}
\begin{proof} Since $\Y\subset \X$, for any host space $H$, 
$D_N(H,\Y)\ge D_N(H,\X)$. On the other hand, if $\sup_{Y\in\Y}c_H(Y)=1$, then
\( 1\le \sup_{X\in\X}c_H(X) \le \sup_{Y\in\Y}c_H(Y) \cdot \sup_{X\in\X}c_\Y(X)= 1 .  \)
\end{proof}

Table~\ref{tab:classes} lists the classes of finite metric spaces which we will deal with and their regular dense subclasses. The proofs of the
density are standard.

\begin{table}[ht]
\begin{center}
\begin{tabular}{l|ll}
Metric class &  Dense structured  & Shorthand\\ 
(Finite subsets of) &  subclass ($n\in\mathbb N$) & ($n\in\mathbb N$) \\ \hline
$\mathbb R$  & $\{0,\ldots,n\}$ & $P_n$\\
$L_1$  &  $(\{0,1\}^n,\|\cdot\|_1)$ & $\{0,1\}^n$ \\
$L_\infty$ ~ (i.e., $\MET$) & $(\{1,\ldots,n\}^n,\|\cdot\|_\infty)$ & $[n]^n_\infty$ \\
tree metrics & $(\{0,1\}^{\le n}, \text{tree distance})$ & $B_n$ \\ \hline
\end{tabular}
\end{center}
\caption{The classes of metric spaces considered in these notes, and their dense regular subclasses used in the proofs. $\MET$ is the class of all finite metric spaces. $\{0,1\}^{\le n}$ is the set of binary strings of length at most $n$. The tree distance on binary strings
$x,y\in \{0,1\}^{\le n}$ is defined as $|x|+|y|- 2 |\mathrm{lcp}(x,y)|$,
where $|x|$ is the length of $x$, and $\mathrm{lcp}(x,y)$ is the longest common prefix of $x$ and $y$.}
\label{tab:classes}
\end{table}

\section{Qualitative dichotomies} \label{sec:qualitative}

\begin{theorem}~\cite{Mat-BD}  \label{thm:mat-bd} 
The following classes of finite metric spaces have the qualitative dichotomy property:
\begin{enumerate}[1.]
\item Finite subsets of $\mathbb R$,
\item For any $p\in [1,\infty]$, the class of finite subsets of $L_p$.
\item Finite equilateral spaces.
\end{enumerate} 
\end{theorem}

Here we just outline the proof for finite subsets of $L_p$, which is a nice demonstration of the linearization technique for Lipschitz mappings of normed spaces: Given a Lipschitz map, find a point of differentiability.
The differential is a linear map with the same Lipschitz norm. Now apply a result from the linear theory. 
In our case, the linear result is the Maurey, Pisier, and Krivine theorem, and the differentiability argument
is due to Kirchheim~\cite{Kirchheim}.

\begin{proof}[Sketch of a proof of Theorem.~\ref{thm:mat-bd}, item 2]
Fix a host space $H$, and $p\ge 1$, and assume that there exists $K\in [1,\infty)$
such that any finite subset of $S\subset L_p$ embeds in $H$, $f_S:S\to H$, and  $\dist(f_S)\le K$.
We fix a finite $S\subset L_p$, $\e>0$ and want to prove that $c_H(S)\le 1+\e$. It is known (see~\cite[Sec.~11.2]{DZ}) 
that $S$ can be isometrically embedded in $\ell_p^t$, for $t=\binom{|S|}{2}$. Let $n=n(p,t,K,\e)$ be chosen as in the Maurey-Pisier+Krivine theorem discussed in Section~\ref{sec:intro}.

The argument of the proof goes roughly as follows: Use a compactness argument to conclude that there exists an
embedding $\hat f:\ell_p^n \to H$ whose distortion at most $K$. Since this embedding is in particular Lipschitz, use a differentiation argument to find a point of ``differentiability". The differential is a ``linear mapping" of $\ell_p^n$ whose distortion is at most $K$,
and thus by the Maurey-Pisier+Krivine theorem alluded to above, there exists a $t$-dimensional subspace of 
$\ell_p^n$ which is almost isometric to $\ell_p^t$ for which that mapping is almost isometry into $H$. This would finish the proof, since $S$ is isometrically embeddable in $\ell_p^t$.

The above argument has two major difficulties: 
\begin{inparaenum}[1)]
\item Since $H$ is not compact, the required ``compactness argument" is false.
\item  The notions of linear mapping and derivative when the target is general metric space, are not clear.
\end{inparaenum}

The first difficulty is addressed using  a compactness argument, similar to Rado's Lemma (see \cite[Lemmas~3.4, \&~4.4]{Mat-BD}), which implies
there exists a metric space $\hat H$, and an embedding 
$\hat f:\ell_p^n\to \hat H$ such that $\dist(\hat f)\le K$, and 
moreover, for any finite $T\subset \ell_p^n$, and $\delta>0$, there exists $R$, $T\subset R\subset \ell_p^n$ such that $\hat f(R)$ distorts the distance in $f_R(R)$ by at most a factor of $1+\e$. 

The second difficulty is addressed using
a metric differentiation theorem of Kirchheim~\cite{Kirchheim}. It implies that 
there exists $x_0\in \mathbb R^n$, and a pseudo-norm on $\tb\cdot \tb$ on $\mathbb R^n$ 
such that for every $h,k\in \ell_p^n$, 
$d_H(\hat f(x_0+h),\hat f(x_0+k))= \tb h-k\tb +o(\|h\|_p+\|k\|_p)$. We conclude that $\tb \cdot \tb $ is a norm on $\mathbb R^n$ whose Banach-Mazur distance from the $\ell_p$ norm is at most $K$. 
Furthermore on a ball $B$ small enough around $x_0$ in $\ell_p^n$, $\tb \cdot \tb $ is $1+\e$ approximation to the metric on $\hat f(B)$.

Hence, by translating and rescaling $S$ we can assume it is inside $B$, and thus we can view $\hat f$ as an approximate mapping between $\ell_p^n$ and $(\mathbb R^n,\tb \cdot \tb)$. At this point Maurey-Pisier+Krivine
theorem can be applied rigorously.
\end{proof}

\medskip
Finite subsets of finite (but larger than one) dimensional normed space is a natural class of metric spaces that does not have the dichotomy property:
\begin{proposition}
Fix $d>1$, and  some norm, $\|\cdot \|$, on $\mathbb R^d$. Then the class of finite subsets of $(\mathbb R^d, \|\cdot\|)$ does not have qualitative dichotomy.
\end{proposition}
\begin{proof}[Sketch of a proof]
It is possible to construct another norm $\tb \cdot \tb $ on $\mathbb{R}^d$, whose Banach-Mazur
distance from $\|\cdot \|$ is some $B>1$.
I.e., for any 
linear mapping $T:(\mathbb R^d, \|\cdot\|) \to (\mathbb R^d, \tb \cdot \tb )$, $\|T\|\cdot \|T^{-1}\|\ge B>1$, 
and the inequality is tight for some $T$.%
\footnote{To see it, notice that the Banach-Mazur distance
between $\ell_2^d$, and $\ell_1^d$ is $\sqrt{d}>1$, and therefore by the triangle inequality any other $d$-dimensional norm must be at distance
at least $\sqrt[4]{d}>1$ from one of them. By John's Theorem (see e.g.~\cite[Sec.~13.4]{Mat-Discrete-Geometry}),  
the distance is at most $d$.} 
We take $H=(\mathbb R^d, \tb \cdot \tb )$, and so $c_H((\mathbb R^d,\|\cdot\|))\le B$.
On the other hand, assume for the sake of contradiction  that there exists $A<B$ such that any finite subset of $(\mathbb R^d, \|\cdot\|)$ can be embedded in $H$ with distortion at most $A$. By a compactness argument there exists an embedding of the unit ball of
$(\mathbb R^d, \|\cdot \|)$ in $(\mathbb R^d, \tb \cdot \tb )$ with distortion at most $A$. Next, by Rademacher differentiation theorem
there exists a point of differentiability in this embedding. The differential is a linear mapping whose  distortion is at most $A$, which is a contradiction.
\end{proof}

\section{Dichotomy for subsets of the line} \label{sec:line}

In the next two sections we discuss quantitative dichotomies and sketch direct proofs. 
We begin with finite subsets of $\mathbb R$.

\begin{theorem} \label{thm:line-dich}
For every metric space $H$, either
\begin{compactitem}
\item $c_{ H}(A) =1$, for every finite $A\subset \mathbb R$; or
\item There exists $\beta>0$, such that $c_{ H}(P_n)\ge \Omega(n^\beta)$, where $P_n$ is the $n$-point path metric. 
\end{compactitem}
\end{theorem}

As discussed above, {\Matousek} showed a qualitative dichotomy for finite subsets of the line, based on 
differentiation argument. His proof actually gives the same quantitative bounds as in Theorem~\ref{thm:line-dich}.
Here, following~\cite{MN-trees}, we sketch a somewhat different proof which conveys the approach to prove the more complicated quantitative dichotomies for finite subsets of $L_1$, and for all finite metric spaces. 

The general approach in those proofs is to define an ``isomorphic" inequality, and prove sub-multiplicativity.
This approach --- originated in the work of Pisier~\cite{Pisier-type-1} --- is used in Banach space theory quite often.

We proceed to prove Theorem~\ref{thm:line-dich}. We first 
choose an appropriate inequality that captures the distortion of embedding $P_n$ in $H$. 
Let $\Psi_n(H)$ be the infimum over $\Psi>0$ such that for every $f:P_n \to H,$
\begin{equation} \label{eq:path}
  d_H(f(0),f({n})) \le \Psi n\max_{i=0,\ldots n-1} d_H(f(i),f({i+1})) .
\end{equation}


\begin{lemma} \label{lem:path-ineq}
For every metric space $H$, and $m,n\in \mathbb N$,
\begin{compactenum}[1.]
\item $\Psi_n( H) \le 1$.
\item $c_{ H}(P_n)\ge 1/ \Psi_n( H)$.
\item If $\Psi_n( H)=1$, then $c_{ H}(P_n)=1$.
\item $\Psi_{mn}( H) \le \Psi_m( H) \cdot \Psi_n( H)$.
\end{compactenum}
\end{lemma}

Before proceeding with the proof of lemma~\ref{lem:path-ineq}, lets see how 
Theorem~\ref{thm:line-dich} is derived.
\begin{proof}[Proof of Theorem~\ref{thm:line-dich}]
We will prove the dichotomy to $(P_n)_n$ (the path metrics), which by Prop.~\ref{prop:dense} is sufficient.
Fix a host space $H$.
\begin{itemize}
\item If for every $n\in\mathbb N$, $\Psi_n(H)=1$, then $c_{H}(P_n)=1$. 
 
\item If there exists $n_0$ for which $\Psi_{n_0}(H)=\eta<1$, then let $\beta>0$ be such that $n_0^{-\beta}=\eta$,
and from the submultiplicativity, $\Psi_{n_0^k}(H)\le \eta^k= (n_0^k)^{-\beta}$, and so
$c_{H}(P_{n_0^k})\ge (n_0^k)^\beta$. \qedhere
\end{itemize}
\end{proof}

\begin{proof}[Proof of Lemma~\ref{lem:path-ineq}]
~
\begin{enumerate}[1.]
\item Follows from the triangle inequality. 
\item 
Fix $f:P_n\to H$, and $\Psi>\Psi_n(H)$. Plugging the Lipschitz norms  into~\eqref{eq:path},
\begin{multline*} \frac{n}{\|f^{-1}\|_\Lip} \le d_H(f(0),f({n})) 
\le \Psi n\max_{i=0,\ldots n-1} d_H(f(i),f({i+1})) \le \Psi n \|f\|_\Lip ,
\end{multline*}
So $\dist(f)=\|f\|_\Lip \cdot \|f^{-1}\|_\Lip \ge 1/\Psi$. Since this is true for any $\Psi>\Psi_n(H)$, $\dist(f)\ge 1/\Psi_n(H)$.

\item If $\Psi_n( H)=1$, then for any $\e\in(0,1/2n)$, there exists $f:P_n \to H$ for which
\begin{multline} \label{eq:path-2}
 n\max_{i=0,\ldots n-1} d_H(f(i),f({i+1})) \ge d_H(f(0),f({n-1})) 
 \\ \ge (1-\e) n\max_{i=0,\ldots n-1} d_H(f(i),f({i+1})). 
 \end{multline}
Let $A= \max_{i=0,\ldots n-1} d_H(f_i,f_{i+1})$,
and for $i> j$,
\begin{multline*} (i-j)A \ge d_H(f(i),f(j)) \\\ge d_H(f(0),f({n})) - d_H(f(j),f(0)) - d_H(f(n),f(i)) \\
\ge (1-\e)n A - j A - (n-i)A= (i-j -\e n) A .
\end{multline*}
This means that $\dist(f)\le 1+2\e n$, which implies that $c_H(P_n)=1$.

\item Fix $f:P_{mn}\to H$. Define $g:P_n \to H$, by $g(i)=f(im)$. 
Applying~\eqref{eq:path} to $g$, we obtain
\begin{equation}\label{eq:path-3}
 d_H(f(0),f({mn})) \le (\Psi_n(H)+\e) n \max_{i=0\ldots n-1} d_H(f({im}),f({(i+1)m})). 
 \end{equation}
Next, define $h_i:P_m\to H$, $h_i(j)=f(im+j)$, and apply~\eqref{eq:path} for each $h_i$, and so
\begin{multline} \label{eq:path-4}
 d_H(f({im}),f({(i+1)m})) \\ \le (\Psi_m(H)+\e) m  \max_{j=0\ldots m-1} d_H(f({im+j}),f({im+j+1})) .
\end{multline}
Combining~\eqref{eq:path-3} with~\eqref{eq:path-4}, and we conclude the claim. \qedhere
\end{enumerate}
\end{proof}

The quantitative dichotomy  in Theorem~\ref{thm:line-dich} is tight for finite subsets of the line: 
For any $\beta\in(0,1]$, there exists $H_\beta$ such that $c_{H_\beta}(P_n)=\Theta(n^\beta)$. For $\beta\in(0,1)$, $H_\beta$ can be taken as the real line with the usual metric to the power of $1-\beta$.
For $\beta=1$, $H_1$ can be taken as the ultrametric defined on $\{0,1\}^{\mathbb N}$, with the distance function
$\rho(x,y)=2^{-|\mathrm{lcp}(x,y)|}$, where $\mathrm{lcp}$ is the longest common prefix of the two sequences.

\section{Dichotomies for finite subsets of $L_1$, and $L_\infty$}

The proofs of the quantitative dichotomies for subsets of $L_1$ and subsets of $L_\infty$ 
use the same general approach taken in Section~\ref{sec:line}: 
we write inequalities for which we can prove a lemma similar to Lemma~\ref{lem:path-ineq}, but replacing paths with Hamming cubes (for subsets of $L_1$) and grids with the $\ell_\infty$ distance (for subsets of $L_\infty$).

In both cases the hard part in the proof seems to be coming up with the inequality. However, in contrast to 
path metrics, the proofs of the lemmas analogous to Lemma~\ref{lem:path-ineq} (especially item~(3)) are technical and lengthy. We will therefore omit all these details and concentrate on the inequalities.

%
%

\subsection{Finite subsets of $L_1$}

The argument given here is essentially from a paper of Bourgain, Milman, and Wolfson~\cite{BMW} on metric type.%
\footnote{The paper~\cite{BMW} does not discuss dichotomy, but rather
a non-linear analogue for Pisier theorem for type-1. As we shall see in Section~\ref{sec:MP}, from that result it is easy to obtain the dichotomy.}

\begin{theorem}{\cite{BMW}} \label{thm:dich-cubes}
For every metric space $H$, either 
\begin{compactitem}
\item $c_{H}(X) =1$, for every finite $X\subset L_1$; or
\item There exists $\beta>0$, such that $c_{H}((\{0,1\}^n,\|\cdot\|_1)\ge \Omega(n^\beta)$.
\end{compactitem}
\end{theorem}

Similarly to the dichotomy of subsets of $\mathbb R$, we use an inequality to guide the proof:
Let $(e_i)_{i=1}^n$ denote the standard basis of $\{0,1\}^n$, and $\mathbf 1=\sum_i e_i$.
Let $T_n(H)$ be the infimum over $T>0$ such that 
for every $f:\{0,1\}^n \to H$,
\begin{equation} \label{eq:cube}
 \Avg_{x\in\{0,1\}^n}
 d_H(f(x),f(x+\mathbf{{1}}))^2 \le T^2 n\sum_{i=1}^n  \Avg_{x\in\{0,1\}^n} d_H(f(x),f(x+e_i))^2 ,
\end{equation}
where the operator $\Avg$ means averaging.

Inequality~\eqref{eq:cube} was chosen to ``capture" the distortion of embeddings the Hamming cubes in $H$, and have
the sub-multiplicativity property (in $n$). It is a variant of the metric-type inequality from~\cite{BMW}. The connection (and motivation) to the type property is expanded upon in Section~\ref{sec:MP}.
Formally, we can prove a lemma analogous to Lemma~\ref{lem:path-ineq}:

\begin{lemma} \label{lem:cube-ineq}
For every metric space $H$, and $m,n\in \mathbb N$,
\begin{enumerate}[1.]
\item $T_m(H) \le 1$.
\item $c_{H}(\{0,1\}^n)\ge 1/ T_n( H)$.
\item If $T_n( H)=1$, then $c_{H}(\{0,1\}^n)=1$.
\item $T_{mn}( H) \le T_m(H) \cdot T_n(H)$.
\end{enumerate}
\end{lemma}

Using Lemma~\ref{lem:cube-ineq}, the proof of Theorem~\ref{thm:dich-cubes}
is  the same as the proof of Theorem~\ref{thm:line-dich},
replacing references to $\Psi_n(H)$ with $T_n(H)$, and the path metric with the Hamming cube.

Regarding the quantitative tightness of Theorem~\ref{thm:dich-cubes}:
It is known~\cite{Enflo-cubes} that $c_{\ell_2}((\{0,1\}^n,\|\cdot\|_1))=\sqrt{n}$, 
and that any $N$-point subset of $L_1$ is $O(\sqrt{\log N} \log\log N)$ embeddable in $\ell_2$~\cite{ARV,CKR,ALN}.
I do not know much more.
\begin{question}
Does there exist $\beta\in (0,1/2)$ and a metric space $H$
such that $1<D_N(H,2^{\ell_1})=O((\log N)^\beta)$? If so, is it true for every $\beta>0$?
\end{question}

\subsection{Finite metric spaces}

Next,
we consider $\MET$, the set of all finite metric spaces 
which is equal to the set of finite subsets of $L_\infty$.
\begin{theorem}\cite{MN-cotype-full} \label{thm:dich-all}
For every metric space $H$, either
\begin{compactitem}
\item $\sup_{X\in \MET} c_{\mathcal H}(X) =1$; or
\item There exists $\beta>0$, such that $c_{H}((\{1,\ldots,n\}^n, \|\cdot\|_\infty)\ge \Omega(n^\beta)$,
where $[n]^n_\infty$ is the $\{1,\ldots,n\}^n$ grid with the $\ell_\infty$ distance.
\end{compactitem}
\end{theorem}

Again, we use an inequality (derived from the metric cotype inequality~\cite{MN-cotype-full}) to guide the proof.
Denote by $\Gamma_{n}(H)$ the infimum over $\Gamma>0$ such that for every $m\in \mathbb N$,
and every $f:\mathbb Z_m^n \to H$,
\begin{equation} \label{eq:grid}
\sum_{i=1}^n \Avg_{x\in \mathbb Z_m^n} d_H(f(x),f(x+n e_j))^2 \le 
\Gamma^2\cdot n^2 \cdot n \Avg_{\e\in \{\pm 1\}^n} \Avg_{x\in \mathbb Z_m^n} d_H(f(x),f(x+\e))^2
\end{equation}
(the additions ``$x+n e_j$", and ``$x+\e$" are in $\mathbb Z_m^n$).
Inequality~\eqref{eq:grid} is designed to capture the distortion of embedding $\{1,\ldots,n\}^n$ with the $\ell_\infty$
distance in $H$. In this context, it seems more natural to average $\e$ over all $\{-1,0,1\}^n$ which are the distance 1 in the $\ell_\infty$ metric. However, this choice would complicate the proof of the submultiplicativity property. Note that the metric induced by the graph $\mathbb Z_n^n$ with the $\{\pm 1\}^n$ edges contains an isometric copy of $\{1,\ldots, n/4\}^n$
with the $\ell_\infty$ metric. Also, the design choice of the universal quantifier on $m$ (instead of say fixing $m=n$) was done to make the proof of the submultiplicativity easy. The connection to the cotype property of Banach spaces is expanded upon in Section~\ref{sec:MP}.
As before, we use a lemma analogous to Lemma~\ref{lem:path-ineq}:
\begin{lemma} \label{lem:torus-ineq}
For every metric space $H$, and $m,n\in \mathbb N$,
\begin{enumerate}[1.]
\item $\Gamma_n( H) \le 1$, for all even $n$.
\item $c_{ H}((\{1,\ldots,n/4\}^n, \|\cdot\|_\infty))\ge 1/ \Gamma_n(H)$.
\item If $\Gamma_n(H)=1$, then $c_{ H}((\{1,\ldots,n/4\}^n, \|\cdot\|_\infty))=1$.
\item $\Gamma_{m n}(H) \le \Gamma_{m}(H) \cdot \Gamma_{n}(H)$.
\end{enumerate}
\end{lemma}

Theorem~\ref{thm:dich-all} is deduced similarly to Theorems~\ref{thm:line-dich} and~\ref{thm:dich-cubes} but now using 
Lemma~\ref{lem:torus-ineq}. A sketch of the proof of Lemma~\ref{lem:torus-ineq} 
can be found in~\cite[Sec.~2]{MN-cotype}. The complete proof appears in~\cite[Sec.~6]{MN-cotype-full}.

I do not know much about the quantitative tightness of Theorem~\ref{thm:dich-all}. 
\begin{question} \label{q:dich-all}
Does there exist $\beta\in(0,1)$ and a metric space $H$
for which $1<D_N(H,\MET)=O((\log N)^\beta)$? If so, is it true for every $\beta\in(0,1)$?
\end{question}
Personally, the dichotomy of $\MET$ seems to me the most natural problem in thess notes, and Question~\ref{q:dich-all} the most fundamental
open problem. 

Bourgain's embedding theorem~\cite{Bourgain-embed} and the matching lower bound~\cite{LLR} implies
that $D_N(\ell_2,\MET)=\Theta(\log N)$.
Essentially, all examples of families of metrics having logarithmic distortion when embedded into Hilbert space, are families of expander graphs with a logarithmic diameter. {\Matousek} proved that expanders have logarithmic distortion when embedded in $L_p$, for any $p\in [1,\infty)$. Recently, Lafforgue~\cite{Lafforgue} has exhibited classes of expanders 
with logarithmic distortion when embedded in $B$-convex Banach spaces
--- spaces with type greater than~$1$.

In view of the seemingly surprising fact of the metric dichotomy, it is natural to ask which monotone 
$f:\mathbb N\to\mathbb N$ has a metric space $H$ such that $D_N(H,\MET)=\Theta^*(f(N))$ (here $\Theta^*$ can ``hide" polylogarithmic multiplicative factors).
From Bourgain embedding theorem we know that $f(N)=\log N$ is achievable using Hilbert space. {\Matousek}~\cite{Mat-lowdim} showed that for every even $d$, 
$D_N(\ell_2^d,\MET)=\Theta^*(N^{2/d})$. Furthermore, in the full version of~\cite{MN-trees}
it is shown that for every $\e\in(0,1]$, there exists a metric space $H_\e$, for which $D_N(H_\e,\MET)=\Theta(N^\e)$.
This leaves us with a concrete question:
\begin{question}
Does there exist $H$ for which $D_N(H,\MET)\in \omega(\log N) \cap N^{o(1)}$? 
\end{question}

Comparing the results in this section with those in Section~\ref{sec:qualitative}, we also ask:

\begin{question}
Is there a (substantial) quantitative dichotomy for finite subsets of
$L_p$, and in particular for $L_2$?
\end{question}

\subsection{Metric type and cotype and non-linear Maurey-Pisier theorems}
\label{sec:MP}

A normed space is said to have type $p$, $1\leq p\leq 2$, with constant $T$, if for every finite family $x_1,\dots,x_n\in X$,
\begin{equation}\label{eq:type}
\Big(\Avg_{\e\in\{-1,1\}^n}\Big\|\sum^n_{j=1}\e_j x_j\Big\|^p_X\Big)^{\frac1p}\leq T\Big(\sum^n_{j=1}\|x_j\|^p\Big)^{\frac 1p}
\end{equation}
and cotype $q$, $2\leq q\leq\infty$, with constant $C$, 
if for every finite family $x_1,\dots,x_n\in X$,
\begin{equation} \label{eq:cotype}
C \Big(\Avg_{\e\in\{-1,1\}^n}\Big\|\sum^n_{j=1}\e_j x_j\Big\|^q_X\Big)^{\frac1q}\geq \Big(\sum^n_{j=1}\|x_j\|^q\Big)^{\frac 1q}
\end{equation}

The theory around these notions was developed since the 70's, with fascinating results. The interested reader may consult~\cite{MS,Maurey-survey} and references there in. 
Here we are interested in one aspect of this theory, called 
Maurey-Pisier Theorem.

A closely related conditions are \emph{equal norms type and cotype}: A normed space is said to have equal norm (en) type $p$, $1\leq p\leq 2$, with constant $\hat T$ if for every finite family $x_1,\dots,x_n\in X$,
\begin{equation}\label{eq:en-type}
\Avg_{\e\in\{-1,1\}^n} \Big\|\sum^n_{j=1}\e_j x_j\Big\|^2_X\leq \hat T^2 n^{\frac 2p-1}\sum^n_{j=1}\|x_j\|^2
\end{equation}
Similarly, a normed space is said to have equal norms (en) cotype $q$, $2\leq q\leq\infty$, with constant $\hat C$, 
if for every finite family $x_1,\dots,x_n\in X$,
\begin{equation} \label{eq:en-cotype}
\hat C^2 n^{1-\frac2q} \Avg_{\e\in\{-1,1\}^n}\Big\|\sum^n_{j=1}\e_j x_j\Big\|^2_X\geq \sum^n_{j=1}\|x_j\|^2
\end{equation}

Type/cotype and en-type/cotype are closely related%
\footnote{The LHS are equivalent, up to a constant factor, by Kahane inequality. In the RHS, type/cotype condition implies the equal norms variant by H\"older inequality, and they are equal when all the $x_i$ has the same norm, hence the name.}:
Type $p$ implies en-type $p$ which implies type $p-\e$ for every $\e>0$. 
Cotype $q$ implies en-cotype $q$ which implies cotype $q+\e$ for every $\e>0$ (see~\cite{T-J}).

We observe that any normed space has en-type 1 and en-cotype $\infty$, as these inequalities follow from the triangle inequality (and Cauchy-Schwarz). 
A normed space has type $>1$ if and only if it has en-type $>1$,
and cotype $<\infty$ if and only if it has en-cotype $<\infty$.
Maurey and Pisier~\cite{MP-cotype-oo} proved: 
\begin{theorem} \label{thm:MP-cotype-oo}
A normed space $X$ does not have (en-)cotype $<\infty$ if and only if for every $n\in \mathbb N$,
and $\eta>0$, $\ell_\infty^n$ can be linearly embedded in $X$ with distortion at most $1+\eta$.
\end{theorem}
Pisier~\cite{Pisier-type-1} proved an analogous result for type:
\begin{theorem} \label{thm:P-type-1}
A normed space $X$ does not have (en-)type $>1$ if and only if for every $n\in \mathbb N$,
and $\eta>0$, $\ell_1^n$ can be linearly embedded in $X$ with distortion at most $1+\eta$.
\end{theorem} 

Those results imply dichotomy results for the class of finite dimensional subspaces.
\begin{proposition} \label{prop:linear-dichotomy}
For any normed space $H$, either
\begin{compactitem}
\item For every $\e>0$, any finite dimensional normed space $X$ can be linearly embedded in $H$ with distortion $1+\e$.
\item There exists $\beta>0$, such that for every $n\in\mathbb N$, 
and linear embedding $f:\ell_\infty^n \to H$, 
$\dist(f)= \Omega(n^{\beta})$.
\end{compactitem}
\end{proposition}
\begin{proof}
If  $H$ does not have finite en-cotype, then by Theorem~\ref{thm:MP-cotype-oo} any $\ell_\infty^n$
can be linearly embedded in $H$ with distortion $1+\e$, for every $\e>0$. Combining this with the elementary observation that for any $\e>0$, any finite dimensional normed space can be $(1+\e)$-embedded in $\ell_\infty^n$, for some $n\in \mathbb N$ we obtain the first bullet.

If, on the other hand, $H$ has some finite en-cotype $q<\infty$, then for any linear embedding $f: \ell_\infty^n \to H$,
\begin{multline*}
\hat C \|f\|= \hat C\|f\| \Big(\Avg \Big\|\sum^n_{j=1}\e_j e_j\|^2_\infty \Big)^{\frac {1}{2}}\ge \hat C \Big(\Avg \Big\|\sum^n_{j=1}\e_j f(e_j)\Big\|^2_H\Big)^{\frac12} \\
\geq n^{\frac1q-\frac12} \Big(\sum^n_{j=1}\|f(e_j)\|_H^2\Big)^{\frac 12}
\ge \|f^{-1}\|n^{\frac1q-\frac12}\Big(\sum^n_{j=1}\|e_j\|_\infty^2\Big)^{\frac 12}= n^{\frac{1}{q}} \cdot \|f^{-1}\|, 
\end{multline*}
which implies that $\dist(f)=\Omega(n^{1/q})$.
\end{proof}
A similar dichotomy can be proved for $\ell_1^n$, using Theorem~\ref{thm:P-type-1}.

Theorems~\ref{thm:MP-cotype-oo} and~\ref{thm:P-type-1} are proved using linear analogues of Lemmas~\ref{lem:cube-ineq}, and~\ref{lem:torus-ineq}.\footnote{The sub-multiplicativity argument originates in the work of Pisier~\cite{Pisier-type-1} on type $1$. The original proof
of the cotype $\infty$ in~\cite{MP-cotype-oo} uses a more complicated argument.}
Indeed,~\eqref{eq:cube} is a natural non-linear analogue of \eqref{eq:en-type}: One can view~\eqref{eq:en-type}
as~\eqref{eq:cube} restricted to \emph{linear} mappings of the cube (and $\hat T$ as a substitute to $Tn^{1-\frac 1p}$). 
Variant of~\eqref{eq:cube} was suggested by Enflo~\cite{Enflo} as a non-linear version of type (Ineq.~\eqref{eq:type}), and the (almost) equivalence to \mbox{(en-)type} was proved in~\cite{BMW,Pisier-type,MN-type}.

The connection between en-cotype~\eqref{eq:en-cotype}, and~\eqref{eq:grid} is less apparent: It is shown
in~\cite{MN-cotype-full}%
\footnote{Actually, the paper~\cite{MN-cotype-full} proves a variant of the above statement: the equivalence between cotype and metric cotype. But the arguments are the same.}
 that the following metric property is equivalent to en-cotype $q$ in Banach spaces:
There exists $\tilde \Gamma>0$ such that for every $n\in \mathbb N$ there exists $m\in \mathbb N$ such that for every 
$f:\mathbb Z_m^n \to H$,
\begin{equation} \label{eq:metric-en-cotype}
\sum_{i=1}^n \Avg_{x\in \mathbb Z_m^n} d_H(f(x),f(x+\tfrac{m}{2} e_j))^2 \le 
\tilde \Gamma^2 m^2 n^{1-\frac{2}{q}} \!\!\!\! \Avg_{\e\in \{0,\pm 1\}^n} \Avg_{x\in \mathbb Z_m^n} d_H(f(x),f(x+\e))^2.
\end{equation}

Inequality~\eqref{eq:metric-en-cotype} is ``close in spirit" to~\eqref{eq:grid}. 
Rigorously, if we change~\eqref{eq:grid} by replacing the ``$n$" in the LHS with ``$n^3$", and the ``$n^2$" on the RHS with ``$n^6$", then it is proved in~\cite{MN-cotype-full} that \eqref{eq:metric-en-cotype} is satisfied for some $q<\infty$ if and only if $\limsup_{n\to \infty} \Gamma_n(H)<1$ (where we $\Gamma_n(H)$ is defined according to the modifications of~\eqref{eq:grid} we have just suggested).

\section{Tree metrics}

The class of finite tree metrics does not have the dichotomy property.
\begin{theorem} \cite{MN-trees} \label{thm:dich-trees}
For any $B> 4$, there exists a metric space $H$ such that
$\sup_T c_{H}(T)=B$, where $T$ ranges over the finite tree metrics.
\end{theorem}

It is natural to ask whether there is a dichotomy between constant distortions and say $\log^* N$ distortions. We believe no such dichotomy exists. For complete binary trees we can prove:

\begin{theorem} \cite{MN-trees} \label{thm:dich-binary-trees}
For any $\delta\in(0,0.001)$, and
for any sequence $s(n)$ satisfying (i) $s(n)$ is non decreasing; (ii) $s(n)/n$ is non-increasing
(iii) $4<s(n) \le O(\delta \log n / \log \log n)$, there exists a metric space $H$ and $n_0$ such that for every $n\ge n_0$, 
$(1-\delta) s(n) \le c_H(B_n) \le s(n)$. Here $B_n$ is the (metric on) unweighted complete binary tree of depth $n$.
\end{theorem}

\begin{question}
Is there an extension of Theorem~\ref{thm:dich-binary-trees} to all tree metrics, with dependence on the size of the metric? For example 
``For every $\delta>0$, and $s(n)$ as in Theorem~\ref{thm:dich-binary-trees}, there exists a metric space $H$ such that for every tree metric $T$, 
$(1-\delta) s(\log N) \le D_N(H,\text{trees}) \le s(\log N)$".
\end{question}

Theorem~\ref{thm:dich-trees} is a corollary of Theorem~\ref{thm:dich-binary-trees}, when substituting $s(n)=B$, and using
the fact that the  complete binary trees are ``dense" in the finite tree metrics, in the sense of Prop.~\ref{prop:dense}.
The rest of the section is devoted to a non-quantitative sketch  of the argument in the proof of Theorem~\ref{thm:dich-trees}. 
The more complicated (and complete) proof of Theorem~\ref{thm:dich-binary-trees} can be found in~\cite{MN-trees}.

\begin{figure}[ht]
\centering
\includegraphics[scale=0.6]{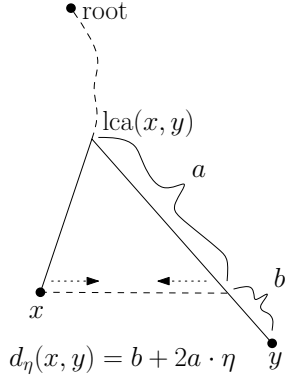}
\caption{The $d_\eta$ metric on $B_\infty$} 
\label{fig:d_eta}
\end{figure}

Denote $\eta=1/B$. To define $H_\eta$, consider 
the infinite binary tree $B_\infty$ with the tree metric on it, and contract the ``horizontal'' distances by a factor of $B$. 
More precisely
Let $h(x)$ be the depth of $x\in B_\infty$, i.e. the distance from the root of $B_\infty$.
Assuming $h(y)\ge h(x)$, the distance
$d_\eta(x,y)$, for $x,y\in B_\infty$ is defined as
\begin{equation} d_\eta(x,y)=h(y)-h(x)+ 2(h(x)-h(\lca(x,y))\cdot \eta. \end{equation}
It is not hard to check that:
\begin{proposition}
\begin{enumerate}
\item $d_\eta$ is a metric on $B_\infty$.
\item $c_{d_\eta}(B_\infty)\le \eta^{-1}$. Indeed the identity mapping does not expand distances, and contracts
them by  factor of at most $1/\eta$.
\end{enumerate}
\end{proposition}

Thus $H_\eta=(B_\infty,d_\eta)$ is our candidate host space for proving Theorem~\ref{thm:dich-trees}. We are left to show that
$\lim_{n\to \infty}c_{H_\eta}(B_n)=\eta^{-1}$.
Our approach follows  {\Matousek}'s proof~\cite{Mat-trees} of: 

\begin{theorem} \label{thm:trees-in-l2}
$C_{L_2}(B_n)\ge \Omega(\sqrt{\log n})$.
\end{theorem}

Bourgain~\cite{Bourgain-trees} proved Theorem~\ref{thm:trees-in-l2} first, 
and there are subsequent proofs~\cite{LS-trees,LNP-markov-convex}.
For our purpose, {\Matousek}'s argument seems the most appropriate, 
we therefore outline his proof of Theorem~\ref{thm:trees-in-l2}.

\begin{figure}[ht]
\centering
\includegraphics[scale=0.9]{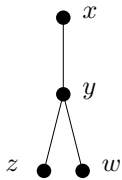}
\caption{A fork $(x,y,z,w)$} 
\label{fig:fork}
\end{figure}

A $\delta$-fork is a quadruple $(x,y,z,w)$ 
such that both $(x,y,z)$, 
and $(x,y,w)$ are $1+\delta$ equivalent to the metric $(0,1,2)$ (where $x$ is mapped to $0$ and $y$ is mapped to $1$). 
It is not hard to see that in Hilbert space (and more generally, $2$-uniform convex spaces), if $(x,y,z,w)$ is a $\delta$-fork then $\|z-w\|\le O(\sqrt{\delta}) \|x-y\| $.

{\Matousek}'s approach is to assume toward a contradiction that there exists a 
Lipschitz embedding $f:B_n \to L_2$ such that $\dist(f)\le c \sqrt{\log n}$, 
and use this assumption to find  a 3-leaf star $(x,y,z,w)$ in $B_n$ whose center is $y$ such that
$(f(x),f(y),f(z),f(w))$ is a $\delta$ fork for $\delta\approx 1/\log n$. 
This implies a large contraction of the distance between $z$ and $w$, which is a contradiction to the assumed upper bound on the distortion.

Consider the first part of {\Matousek}'s proof: Finding a star in $B_n$ whose image is $\delta$-fork. It is proved along the following lines:
Call a metric embedding $f:B_n\to H$, $A$-vertically faithful if $\|f\|_{\Lip}\le 1$, and for every $x,y\in B_n$ in which $x$ is an ancestor of $y$, $d_X(f(x),f(y))\ge d_{B_n}(x,y)/A$.
It turns out (as proved by {\Matousek}) that when considering only the vertical distances in $B_n$, this class has the BD-Ramsey property, or
the dichotomy property. In other words:
\begin{lemma} \label{lem:Bn-BD-Ramsey}
 For every $t\in \mathbb N$, $\delta>0$, and $A>1$, there exists $n=n(t,\delta,A)$, such that
for any host space $H$, and $A$-vertically faithful embedding $f:B_n\to H$,
there exists a subset $C\subset B_n$ which is isometric to $B_t$, and $f(C)$ is $1+\delta$-vertically faithful to $B_t$. 
\end{lemma}

Note that for $t=2$, $f(C)$ contains a copy of $\delta$-fork (actually, two copies). We should also mention that,
not surprisingly, the (simple) proof of Lemma~\ref{lem:Bn-BD-Ramsey} uses the BD-Ramsey property of the path metrics as proved in Section~\ref{sec:line}.

Since the part of finding a $\delta$-fork is independent of the range of the embedding, it makes sense to use it on embedding into $(B_\infty,d_\eta)$.
Examining possible $\delta$-forks $(x,y,z,w)$ inside $H_\eta$,
the two configurations in Fig.~\ref{fig:tip-contract} contracts the distance between $z$ and $w$ by at least $1/(O(\delta)+\eta)$ factor, which
is what we are looking for.

\begin{figure}[ht]
\centering
\includegraphics[scale=0.7]{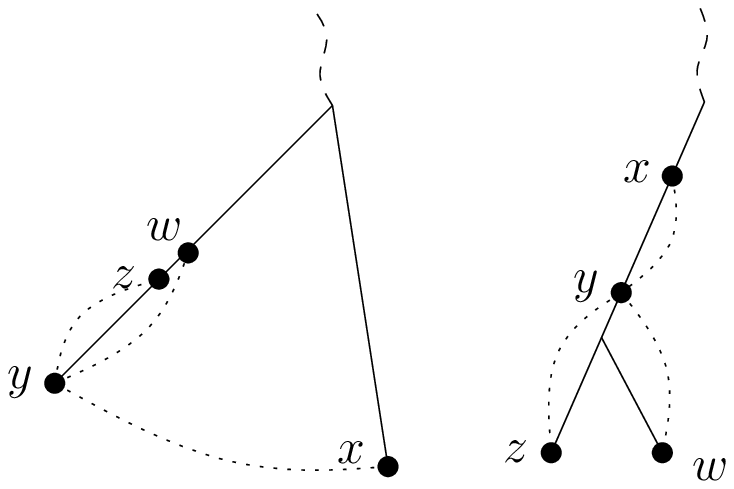}
\caption{\emph{Forks $(x,y,z,w)$ with the distance between the prongs ($z$ and $w$) contracted}. On the right $x$ is an ancestor of the forking point $y$, which is an ancestor to the prongs $z$ and $w$.}
\label{fig:tip-contract}
\end{figure}

However this is not the whole story! There are other types of $\delta$ forks embedded in $H_\eta$. For example type~$II$ in
Fig.~\ref{fig:fork-types}, can be even made 0-fork, but with very small contraction of the tips. This means that the approach that attempts to show large contraction of $\delta$-forks in $H_\eta$ will not work.
There are also other configurations of ``bad" $\delta$-forks, such as type $I$, $III$, and $IV$ in 
Fig.~\ref{fig:fork-types}.%
\footnote{A more careful examination reveals that the configurations labeled type $I$, $III$, $IV$ in Fig.~\ref{fig:fork-types} can be
$O(\eta)$-fork at best. Hence, by taking $\delta\ll \eta$, we can rule out their existence as $\delta$-forks. This approach, however, will fail to prove the more general result of Theorem~\ref{thm:dich-binary-trees}, in which $\eta$ is no longer a constant.}

\begin{figure}[ht]
\centering
\includegraphics[scale=0.7]{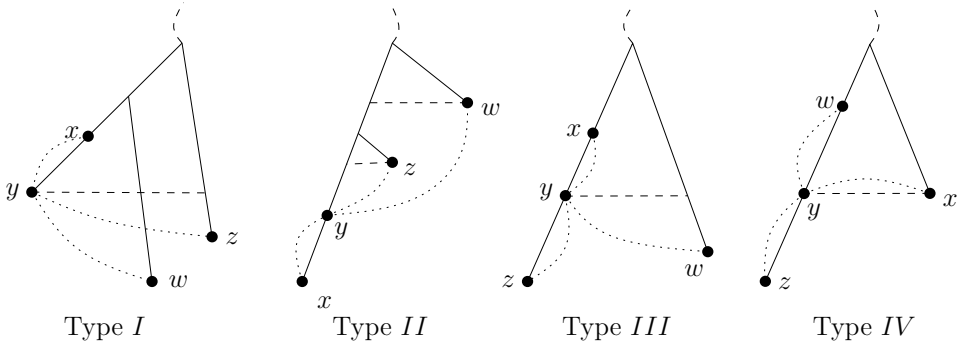}
\caption{\emph{Forks in which the distance between the prongs ($z$ and $w$) do not contract}. In type $II$, $x$ is a descendant of the forking point $y$, which is deeper (in $B_\infty$) than the prongs $z$ and $w$.}
\label{fig:fork-types}
\end{figure}

It turns out that the situation is not that bad. The four types of 
``bad forks" are the only ones that exist. 
\begin{lemma} \label{lem:fork-types}
Every $\delta$-fork in $H_\eta$ is close to one of the 6 types of forks in Figures~\ref{fig:tip-contract} 
and~\ref{fig:fork-types}, up to distortion of $1+O(\delta)$.
\end{lemma}

The proof of Lemma~\ref{lem:fork-types} is a tedious and contains long case analysises (not to mention the need to properly define 
the configurations in Figures~\ref{fig:tip-contract} 
and~\ref{fig:fork-types}). But having it, 
it is reasonable to assume that a slight
generalization of the tip contraction argument for $\delta$-fork would be true in $H_\eta$. Indeed, we show that
\begin{lemma} \label{lem:B_4-nonembed}
Any $1+\delta$ vertically faithful embedding of $B_4$ in $H_\eta$, must
have distortion at least $1/(O(\delta)+\eta)$.
\end{lemma}

Notice that this lemma is sufficient to prove a lower bound 
on the distortion of $B_n$ in $H_\eta$, by using
Lemma~\ref{lem:Bn-BD-Ramsey} with $t=4$.

In order to prove Lemma~\ref{lem:B_4-nonembed}, we view $1+\delta$ vertically faithful embedding of $B_4$ 
as a collection of $\delta$-forks ``glued" together in prong-to-handle fashion.
For this purpose,
it is helpful to analyze what are the possible configurations of $1+\delta$
embedding of 4-point paths, $\{0,1,2,3\}$, in $H_\eta$. There are 
essentially only three different configurations, as depicted in Fig.~\ref{fig:3path-types}.

\begin{figure}[ht]
\centering
\includegraphics[scale=0.7]{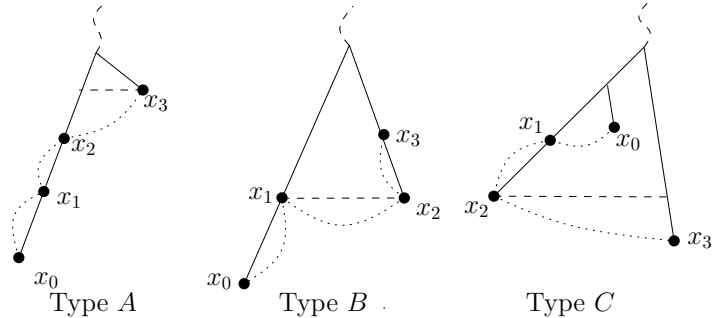}
\caption{\emph{The possible configurations of 4-point path $(x_0,x_1,x_2,x_3)$.} In type $B$, for example, 
$x_1$ is an ancestor of $x_0$,  
$x_3$ is an ancestor of $x_2$, and $x_1$ and $x_2$ have the same depth in $B_\infty$.}
\label{fig:3path-types}
\end{figure}

At this point we can do a ``syntactic" case analysis of how
the $\delta$-forks of Fig.~\ref{fig:fork-types} can be glued together into $1+\delta$ vertically faithful 
embedding of $B_4$, using the ``rules" enforced by the configuration of 4-point paths 
described in Fig.~\ref{fig:3path-types}. Doing this lead to the inevitable conclusion that a fork of a type described in Fig.~\ref{fig:tip-contract}
must appear in the embedding of $B_4$, leading to the conclusion that 
the embedding of $B_4$ must have a large contraction, and hence a large distortion. \hfill \qed

\subsection*{Acknowledgments}
This work was supported by an  Israel Science Foundation (ISF)
grant no.~221/07, and a US-Israel Bi-national Science Foundation (BSF) grant no.~2006009.

The author thanks Assaf Naor for his help in assimilating the subject while collaborating on the papers~\cite{MN-cotype-full,MN-trees}. He also thanks Assaf Naor and the anonymous referee for commenting 
on an earlier version of these notes, which helped improving the presentation.
The figures in these notes are adapted from~\cite{MN-trees}.

Finally, the author wish to thank the organizers of 
the ICMS ``Geometry and Algorithms  workshop", (Edinburgh, 4/2007) and
the organizers of  ``Limits of graphs in group theory and computer science semester" in the Bernoulli Center (Lausanne 5/2007) for inviting him to give a talk on which these notes are based.

\bibliographystyle{plain}
\bibliography{dich}
\end{document}